\documentclass[11pt]{amsart}

\usepackage[pdftex]{graphicx}
\usepackage[dvipsnames]{xcolor}
\usepackage[margin=2.4cm]{geometry}
\usepackage{amsfonts}
\usepackage{amsmath}
\usepackage{amssymb}
\usepackage{amsthm}
\usepackage{mathtools}
\usepackage{enumitem}
\usepackage[colorlinks,cite color=blue,pagebackref=true,pdftex]{hyperref}
\usepackage[T1]{fontenc}
\usepackage{subcaption}
\usepackage{cleveref}

\newcommand\Def[1]{\textbf{\color{black}#1}}

\newcommand\N{\mathbb{N}}
\newcommand\R{\mathbb{R}}

\newcommand\WW{\mathcal W}
\newcommand\MW{\widehat{\WW}}

\DeclareMathOperator{\vol}{vol}

\newcommand\lxrightarrow[1]{\xrightarrow{\raisebox{-2pt}[0pt][0pt]{$\scriptstyle#1$}}}
\newcommand\ixrightarrow[1]{\overset{\raisebox{-2pt}[0pt][0pt]{$\scriptscriptstyle#1$}}\to}
\newcommand\op{\mathrm{op}}
\newcommand\onehat{\hat{1}}
\newcommand\zerohat{\hat{0}}

\newcommand\Ord{\mathcal{O}}
\newcommand\Chain{\mathcal{C}}

\usepackage{tikz}
\usetikzlibrary{cd}
\usetikzlibrary{arrows.meta}
\tikzcdset{arrow style=tikz}
\tikzset{>/.tip={Straight Barb[angle=90:2.5pt 1]}}
\usetikzlibrary{matrix}
\usetikzlibrary{backgrounds}
\usetikzlibrary{calc}
\usetikzlibrary{positioning}
\usetikzlibrary{decorations}
\usetikzlibrary{decorations.markings}
\usetikzlibrary{decorations.pathreplacing}
\usetikzlibrary{decorations.pathmorphing}
\usetikzlibrary{shapes}
\usetikzlibrary{arrows}
\usetikzlibrary{patterns}
\usetikzlibrary{babel}
\tikzset{posetelm/.style={draw, fill, circle, minimum size=4pt, inner sep=0}}
\tikzset{posetelmm/.style={draw, thick, minimum size=5pt, inner sep=0}}
\tikzset{marking/.style={red}}
\tikzset{elmname/.style={blue}}
\tikzset{netedge/.style={thick,-latex}}
\makeatletter
\tikzset{
    dot diameter/.store in=\dot@diameter,
    dot diameter=1.2pt,
    dot spacing/.store in=\dot@spacing,
    dot spacing=6pt,
    dots/.style={
        line width=\dot@diameter,
        line cap=round,
        dash pattern=on 0pt off \dot@spacing,
        shorten <=2pt,
        shorten >=2pt
    }
}
\makeatother

\newtheorem{thm}{Theorem}[section]
\newtheorem{cor}[thm]{Corollary}
\newtheorem{lem}[thm]{Lemma}
\newtheorem{prop}[thm]{Proposition}

\theoremstyle{definition}
\newtheorem{definition}[thm]{Definition}
\newtheorem{example}[thm]{Example}
\newtheorem{rem}[thm]{Remark}
\newtheorem{quest}[thm]{Question}

\title[On PL-homeomorphisms between distributive and anti-blocking polyhedra]{%
On piecewise-linear homeomorphisms between\\ distributive and anti-blocking polyhedra}

\author{Christoph Pegel}
\address{Institut f\"ur Algebra, Zahlentheorie und Diskrete Mathematik, Leibniz Universit\"at Hannover, Germany}
\email{pegel@math.uni-hannover.de}

\author{Raman Sanyal}
\address{Institut f\"ur Mathematik, Goethe-Universit\"at Frankfurt, Germany}
\email{sanyal@math.uni-frankfurt.de}



\date{\today}

\begin{document}

\begin{abstract}
    Stanley (1986) introduced the order polytope and chain polytope of a
    partially ordered set and showed that they are related by a
    piecewise-linear homeomorphism. In this paper we view order and chain
    polytopes as instances of distributive and anti-blocking polytopes,
    respectively. Both these classes of polytopes are defined in terms of the
    componentwise partial order on $\R^n$. We generalize Stanley's
    PL-homeomorphism to a large class of distributive polyhedra using infinite
    walks in marked networks.
\end{abstract}

\maketitle

\section{Introduction}\label{sec:intro}

Let $(P,\preceq)$ be a finite partially ordered set (\Def{poset}, for short).
Stanley~\cite{Stanley} introduced two convex polytopes associated to $P$, the
\Def{order polytope} 
\begin{equation}\label{eqn:Ord}
    \Ord(P) \ \coloneqq \ \left\{ f \in \R^{P} : 
    \begin{array}{cl} 
        0 \ \le \ f(a) \ \le \ 1 & \text{ for all } a \in P\\
        f(a) \ \le \ f(b) & \text{ for all } a \prec b \\
    \end{array} \right\}
\end{equation}
and the \Def{chain polytope} 
\begin{equation}\label{eqn:Chain}
    \Chain(P) \ \coloneqq \ \left\{ g \in \R^{P} : 
    \begin{array}{cl} 
        g(a) \ \ge \ 0 & \text{ for all } a \in P\\
        g(a_1) + g(a_2) + \cdots + g(a_k) \ \le \ 1
         & \text{ for all } a_1 \prec a_2 \prec \cdots \prec a_k
    \end{array} \right\} \, .
\end{equation}
The poset can be completely recovered from $\Ord(P)$ and many geometric
properties of $\Ord(P)$ can be translated into combinatorial properties of
$P$. In particular, the Ehrhart polynomial of $\Ord(P)$ is the order
polynomial of $P$ and the normalized volume $(|P|)! \cdot \vol(\Ord(P))$
is the number of linear extensions of $P$. We refer the reader to Stanley's
original paper and~\cite[Ch.~6]{crt} for more details.  So it is fair to say
that the order polytope $\Ord(P)$ gives a geometric representation of $P$.
The chain polytope, on the other hand, is defined in terms of the
\Def{comparability graph} $G_P = (P,E)$ of $P$. Two elements $a,b \in
P$ share an edge in $G$ if and only if $a \prec b$ or $b \prec a$. Chains in
$P$ correspond to cliques in $G$. The comparability graph can be recovered
from $\Chain(P)$ but $P$ is in general not determined by $G_P$.  Stanley
defines a piecewise-linear (PL) homeomorphism $\phi \colon \R^P \to \R^P$
called the \Def{transfer map} that is volume- and lattice preserving and that
maps $\Ord(P)$ to $\Chain(P)$.  This shows, quite unexpectedly, that both
polytopes have the same Ehrhart polynomial and normalized volume and,
consequently, that order polynomial and number of linear extensions only depend
on the comparability graph.  Order and chain polytopes have many applications
in combinatorics as well as in geometry and, together with their connecting
PL-homeomorphism, have been generalized to marked posets~\cite{ABS,JS,Pegel},
to marked chain-order polytopes~\cite{XinFourier,FFLP}, and to double
posets~\cite{CFS}, to name a few. The aim of this paper is to give a
generalization of Stanley's transfer map to a larger class of geometric objects
that we now define.

Let $V$ be some finite set and $\R^V$ equipped with the usual componentwise
partial order $\le$. A convex polyhedron $Q \subseteq \R^V_{\ge 0}$ is called
\Def{anti-blocking}~\cite{Ful71} or a \Def{convex
corner}~\cite{brightwell} if for $y \in Q$ and $x \in \R^V_{\ge0}$
\begin{equation}\label{eqn:AB}
    x \le y \quad \Longrightarrow \quad x \in Q \,
    .
\end{equation}
The chain polytope is easily seen to be anti-blocking. An \Def{order ideal} in
a poset is a subset that is down-closed with respect to the partial order.
Condition~\eqref{eqn:AB} thus states that anti-blocking polyhedra can be
viewed as \Def{convex order ideals} in $(\R^n_{\ge0},\le)$.

For $x,y \in \R^V$, let us write $x \wedge y$ and $x \vee y$ for the
coordinate-wise minimum and maximum, respectively. In particular,
$(\R^V,\wedge,\vee)$ is an (infinite) distributive lattice with meet $\wedge$
and join $\vee$. It is straightforward to verify that $\Ord(P)$ is closed
under meets and joins. Thus $\Ord(P)$ is a polyhedron as well as a sublattice
of $\R^P$. Such polyhedra were introduced by Felsner and Knauer~\cite{FK11}
under the name \Def{distributive polyhedra}.  Felsner and Knauer noted that
order polytopes and, more generally, alcoved polytopes~\cite{LP} are
distributive. Since marked order polytopes are coordinate sections of dilated
order polytopes, they are automatically distributive. There are many other
polyhedra in combinatorics that turn out to be distributive. For example, the
$t$-Cayley and $t$-Gayley polytopes of Konvalinka and Pak~\cite{KP1}, the
$s$-lecture hall polytopes and cones of Bousquet-M\'elou and
Eriksson~\cite{BE97, BE97a}, and their poset generalizations due to
Br\"and\'en--Leander~\cite{BL}. See \Cref{sec:apps} for more on these
classes of examples.

Stanley's piecewise-linear homeomorphism connects the distributive polytopes
$\Ord(P)$ to the anti-blocking polytope $\Chain(P)$ with phenomenal
combinatorial consequences. Similar PL-maps have been constructed in other
contexts. For example, the polytope $P_n(x)$ studied by Pitman--Stanley is an
anti-blocking polytope and a linear isomorphism to a distributive polytope is
constructed in~\cite[Sect.~4]{PS}. Beck, Braun, and Le~\cite{BBL} introduced
\Def{Cayley polytopes} $C_n$ (denoted by $\mathbf{A}_n$
in~\cite{KP2}) as 
\[
    C_n \ = \ \{ x \in \R^n : 
    1 \le x_i \le 2 x_{i-1} \text{ for all } 1 \le i \le n \}\, ,
\]
where $x_0 \coloneqq 1$. This is a distributive polytope.  In~\cite{KP2},
Konvalinka and Pak define an anti-blocking polytope $\mathbf{Y}_n$ as the set
of all $y \in \R^n$ with $y \ge 0$ and for all $1 \le h \le n$
\[
    \sum_{j=1}^h 2^{h - j} y_j \le 2^{h} - 1
\]
and a linear lattice-preserving map $\phi \colon \R^n \to \R^n$ with
$\phi(\mathbf{Y}_n) = C_n$ to give a simple proof of a conjecture of Braun on
partitions~\cite{BBL}. 

In this paper, we study the relation between distributive and anti-blocking
polyhedra more closely and we construct PL-homeomorphisms for a large class of
distributive polyhedra that subsumes marked order polyhedra. Our PL-maps
generalize Stanley's original construction as well as the mentioned examples
and depends on the convergence of series given by infinite walks in directed
networks.  Most of the work presented here also appeared in the first authors
PhD~thesis~\cite{ThesisPegel}.

\textbf{Acknowledgements.} The second author wants to thank Kolja Knauer and
Martin Skutella for fruitful discussions.

\section{Distributive polyhedra and marked networks}\label{sec:distr}

A \Def{marked network} is a tuple $\Gamma = (V\uplus A, E, \alpha, c,
\lambda)$. It consists of a finite loop-free directed multigraph $(V \uplus A,
E)$ on nodes $V\uplus A$ with edges $E$. We refer to the nodes in $A$ as
\Def{marked nodes} with  \Def{marking} $\lambda \in \R^A$. To every directed
edge $v\lxrightarrow{e}w$ there are two associated weights $\alpha_e, c_e \in
\R$ with $\alpha_e > 0$. In drawings of a marked network, we will depict an
edge $v\lxrightarrow{e}w$ with weights $\alpha_e$ and $c_e$ as
\begin{equation*}
    \begin{tikzpicture}[xscale=.6,baseline=(z.south)]
            \path (-1,0) node[posetelm] (A) {} node[left=2pt,elmname] {\(v\)};
            \path (1,0) node[posetelm] (B) {} node[right=2pt,elmname] {\(w\)};
            \draw[netedge] (A) to node[midway, above=-2pt] (z) {\(\scriptstyle \alpha_e,c_e\)} (B);
    \end{tikzpicture},
\end{equation*} where blue labels are node names. Marked nodes are drawn as
squares with red labels and when edge weights are omitted, we always assume
$\alpha_e=1$ and $c_e=0$. See~\Cref{fig:cayley-network,subfig:bad-network}.

To a marked network,
we associate the polyhedron $\Ord(\Gamma) \subseteq \R^{V}$ consisting of all
points $x \in \R^{V}$ such that
\begin{equation}\label{eqn:distr_ineqs}
    \alpha_e x_w + c_e \ \le \ x_v \quad 
    \text{for all edges } v\lxrightarrow{e}w \, ,
\end{equation}
where we set $x_v\coloneqq \lambda_v$ for $v\in A$.

\begin{example}[Marked order polyhedra]\label{ex:poset}
    For a poset $(P,\preceq)$ let $\widehat{P} = P \uplus \{\zerohat, \onehat\}$
    be the poset with minimum $\zerohat$ and maximum $\onehat$. A marked
    network is obtained from the Hasse diagram of $\widehat{P}$ with $A
    \coloneqq \{\zerohat, \onehat\}$, $V\coloneqq P$, and $E$ consisting of
    edges $v\to w$ for $w$ covered by $v$. Setting $\alpha \equiv 1$, $c
    \equiv 0$ and $(\lambda_{\zerohat}, \lambda_{\onehat}) = (0,1)$, we obtain
    the order polytope $\Ord(P)$.  By allowing more general $A$, this yields
    the marked order polyhedra~\cite{ABS,Pegel}. 
\end{example}

In a similar fashion one sees that the Cayley polytope \(C_n\) is also of the form
$\Ord(\Gamma)$ for the simple network given in \Cref{fig:cayley-network}.
It is straightforward to verify that $\Ord(\Gamma)$ is a distributive
polyhedron. The main result in~\cite{FK11} is a characterization of
distributive polyhedra in terms of marked networks.

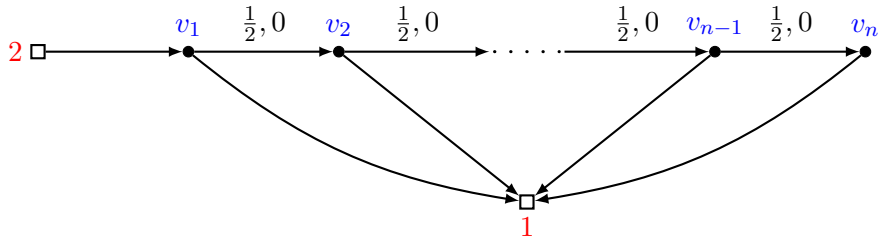
\begin{figure}[ht]
    \centering
    \begin{tikzpicture}[x=2cm]
        \path (0,0) node (2) [posetelmm] {} node [marking,left=2pt] {\(2\)};
        \path (1,0) node (v1) [posetelm] {} node [elmname,above=2pt] {\(v_1\)};
        \path (2,0) node (v2) [posetelm] {} node [elmname,above=2pt] {\(v_2\)};
        \path (4.5,0) node (vn-1) [posetelm] {} node [elmname,above=2pt] {\(v_{n-1}\)};
        \path (5.5,0) node (vn) [posetelm] {} node [elmname,above=2pt] {\(v_n\)};
        \path (3.25,-2) node (1) [posetelmm] {} node [marking,below=2pt] {\(1\)};

        \begin{scope}[every path/.style={netedge}]
            \draw (2) -- (v1);
            \draw (v1) to node[midway,above] {\(\tfrac 1 2, 0\)} (v2);
            \draw (v2) to node[midway,above] {\(\tfrac 1 2, 0\)} (3,0);
            \draw (3.5,0) to node[midway,above] {\(\tfrac 1 2, 0\)} (vn-1);
            \draw (vn-1) to node[midway,above] {\(\tfrac 1 2, 0\)} (vn);
            \draw (v1) to[bend right=15] (1);
            \draw (v2) -- (1);
            \draw (vn-1) -- (1);
            \draw (vn) to[bend left=15] (1);
        \end{scope}
        \draw[thick,dots] (3,0) -- (3.5,0);
    \end{tikzpicture}
    \caption[Marked network for the Cayley polytope $C_n$]{The marked network defining the Cayley polytope $C_n$.}
    \label{fig:cayley-network}
\end{figure}

\begin{thm}[{\cite[Thm.~4]{FK11}}] Every distributive polyhedron is of the
    form $\Ord(\Gamma)$ for some marked network $\Gamma$.
\end{thm}

\begin{rem}
    In order to make the structural similarity to (marked) order polyhedra
    more explicit, our definition of marked network is slightly different from
    that employed in~\cite{FK11}. Most notably, we use markings and require
    \(\alpha_e>0\) instead of allowing loops and our
    edge weights $(\alpha_e,c_a)$ translate to $(\frac{1}{\alpha_e},
    \frac{-c_e}{\alpha_e})$ in the notation of~\cite{FK11}.
\end{rem}

To a marked network with at least all sinks marked, we associate the
\Def{transfer map} $\phi_\Gamma \colon \R^{V} \to \R^{V}$ defined as
\begin{equation}\label{eqn:phi}
    \phi_\Gamma(x)_v  \ \coloneqq \  x_v - \max_{v\ixrightarrow{e}w} \
    (\alpha_e x_w + c_e) \, .
\end{equation}
Let us point out again, that $x_v = \lambda_v$ for $v \in A$.

If $\Gamma$ is derived from a poset $P$ as in \Cref{ex:poset}, the map
$\phi_\Gamma$ is the original transfer map from~\cite{Stanley}. 
If $\Gamma$ is \Def{acyclic}, that is, the underlying directed graph has 
no directed cycles, then we will
see in \Cref{thm:inverse-transfer} that $\phi_\Gamma$ is bijective. In
the non-acyclic situation, this need not be true. In order to illustrate, let
us give a geometric reformulation of the transfer map.  We denote the
standard basis of $\R^V$ by $\{e_v\}_{v \in V}$. For a polyhedron $Q \subseteq
\R^V$ that does not contain $-e_v$ in its recession cone for all $v\in V$, define
the map
$\phi_Q \colon Q \to \R^V$ by
\begin{equation}\label{eqn:geom_phi}
    \phi_Q(x)_v \ \coloneqq \ \max( \mu \ge 0 : x - \mu e_v \in Q )  
\end{equation}
for all $v \in V$. If $Q$ is defined by linear inequalities of the form
$\ell_i(x) \le b_i$, then 
\[
    \phi_Q(x)_v \ = \  \min \left( \tfrac{b_i - \ell_i(x)}{-\ell_i(e_v)} :
    \text{ for } 
    i \text{ with }
    \ell_i(e_v) < 0 \right)  \, .
\]
From~\eqref{eqn:distr_ineqs}, we conclude that $\phi_{\Ord(\Gamma)}$
coincides with $\phi_\Gamma$.

\begin{example} \label{ex:cyclic-non-injective}
    Let \(\Gamma\) be the marked network depicted in \Cref{subfig:bad-network}.
    \begin{figure}
        \centering
        \subcaptionbox[]{Network \(\Gamma\)\label{subfig:bad-network}}[0.18\textwidth][c]{
        \raisebox{2em}{
        \begin{tikzpicture}[xscale=.6,yscale=1.5]
            \path (0,0) node[posetelmm] (C) {} node[below=2pt,marking] {\(0\)};
            \path (1,1) node[posetelm] (A) {} node[right=2pt,elmname] {\(w\)};
            \path (-1,1) node[posetelm] (B) {} node[left=2pt,elmname] {\(v\)};
            \draw[netedge] (A) to[bend left=10] node[midway, below=-2pt] {\(\scriptstyle 2,-2\)} (B);
            \draw[netedge] (B) to[bend left=10] node[midway, above=-2pt] {\(\scriptstyle 2,-2\)} (A);
            \draw[netedge] (A) to[bend left=30] (C);
            \draw[netedge] (B) to[bend right=30] (C);
        \end{tikzpicture}}} \hfill
        \subcaptionbox[]{the polytope \(\Ord(\Gamma)\)\label{subfig:bad-d}}[0.38\textwidth][c]{
        \begin{tikzpicture}[scale=2]
            \draw (-.1,0) -- (2,0) node [right] {\(x_v\)};
            \draw (0,-.1) -- (0,2) node [above] {\(x_w\)};
            \fill[black!6,draw=black,thick] (0,0) -- (0,1) -- (2,2) -- (1,0) -- cycle;
            \draw[blue,dashed, thick] (1,0) -- (1,1) -- (0,1);
            \draw[blue,very thick] (1.5,1) -- (1,1) -- (1,1.5);
            \draw[blue,dashed, thick] (1.6667,1.3333) -- (1,1) -- (1.3333,1.6667);
            \path (.75,1.2) node[fill,circle,inner sep=1pt] (x) {} node[right=-1pt] {\(\scriptstyle x\)};
            \draw[thick,dotted] (0.4,1.2) -- (x) -- (.75,0);
            \draw (0,1) node[left] {\(1\)};
            \draw (1,0) node[below] {\(1\)};
        \end{tikzpicture}} \hfill
        \subcaptionbox[]{the image \(\phi_\Gamma(\Ord(\Gamma))\)\label{subfig:bad-image}}[0.38\textwidth][c]{
        \begin{tikzpicture}[scale=2]
            \draw (-.1,0) -- (2,0) node [right] {\(x_v\)};
            \draw (0,-.1) -- (0,2) node [above] {\(x_w\)};
            \fill[black!6,draw=black,thick] (0,0) -- (0,1.5) -- (1,1) -- (1.5,0) -- cycle;
            \draw[blue,dashed, thick] (1,0) -- (1,1) -- (0,1);
            \draw[blue,very thick] (1.5,0) -- (1,1) -- (0,1.5);
            \path (.35,1.2) node[fill,circle,inner sep=1pt] (x) {} node[below
            right=-1.5pt] (xx) {\(\scriptstyle \phi_\Gamma(x)\)};
            \draw[thick,dotted] (0.0,1.2) -- (x) -- (.35,0);
            \draw (0,1) node[left] {\(1\)};
            \draw (1,0) node[below] {\(1\)};
        \end{tikzpicture}}
        \caption[Marked network and associated polytopes from \Cref{ex:cyclic-non-injective}]{The marked network \(\Gamma\) of \Cref{ex:cyclic-non-injective} with the associated distributive polytope and its ``folded'' image under the non-injective transfer map.}
        \label{fig:bad-example}
    \end{figure}
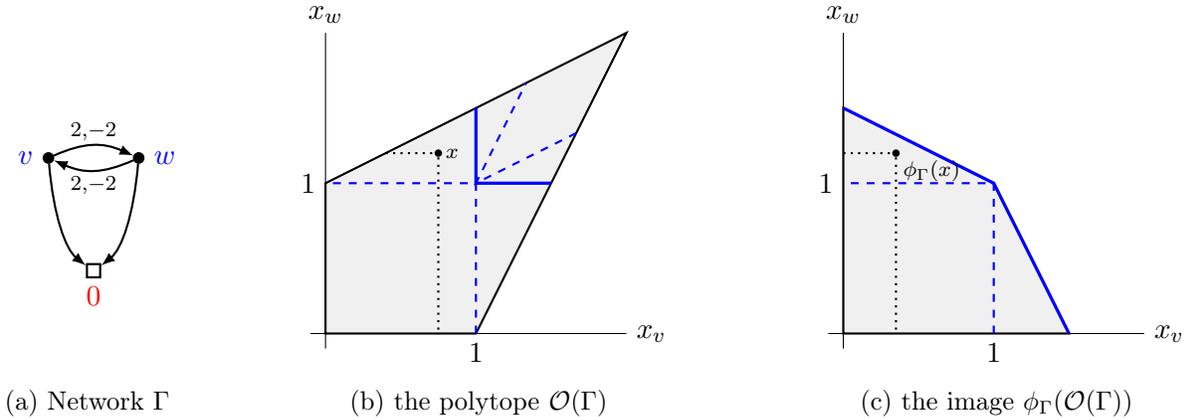
    The distributive polyhedron \(\Ord(\Gamma)\) is a ``kite'' given by the
    inequalities \(0\le x_v\), \(0\le x_w\), \(2x_v-2\le x_w\) and \(2x_w-2\le
    x_v\) as shown in \Cref{subfig:bad-d}.  The transfer map for this network
    is given by
    \begin{equation*}
        \phi_\Gamma \begin{pmatrix} x_v \\ x_w\end{pmatrix} \ = \
        \begin{pmatrix} x_v - \max\{0,2x_w-2\} \\ x_w - \max\{0,2x_v-2\} \end{pmatrix}.
    \end{equation*}
    The transfer map is not injective on \(\Ord(\Gamma)\). For example the
    vertices \((0,0)\) and \((2,2)\) both get mapped to the origin.  In fact,
    the map is $2$-to-$1$ and ``folds'' the polytope along the thick blue line
    in \Cref{subfig:bad-d}. The dashed lines in the lower left part stay fixed
    under the transfer map and have the same image as the dashed lines in the
    upper right part. The geometric behavior of the transfer map given above
    is shown for some \(x\in\Ord(\Gamma)\) using dotted lines.
\end{example}

\begin{example} \label{ex:cyclic-injective}
    Let \(\Gamma\) be the marked network depicted in \Cref{subfig:good-network}.
    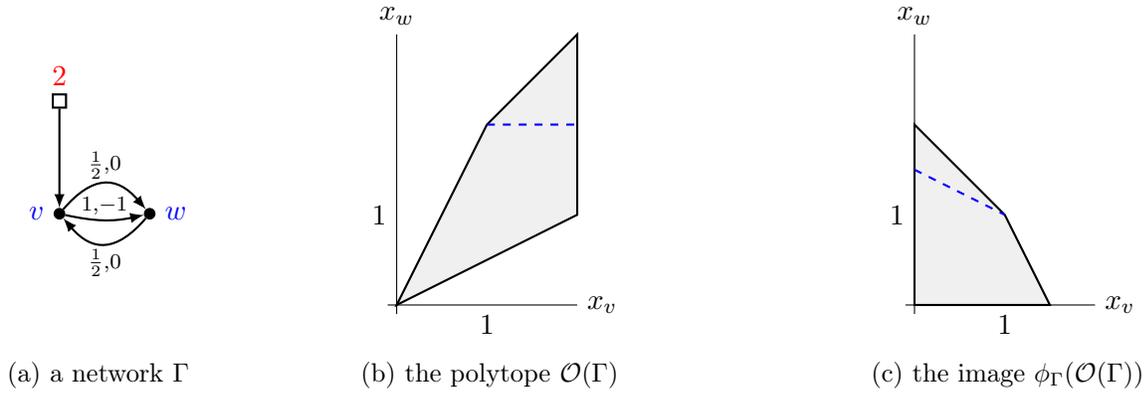
\begin{figure}
        \centering
        \subcaptionbox[]{a network \(\Gamma\)\label{subfig:good-network}}[0.18\textwidth][c]{
        \raisebox{2em}{
        \begin{tikzpicture}[xscale=.6,yscale=1.5]
            \path (-1,1) node[posetelmm] (C) {} node[above=2pt,marking] {\(2\)};
            \path (-1,0) node[posetelm] (v) {} node[left=2pt,elmname] {\(v\)};
            \path (1,0) node[posetelm] (w) {} node[right=2pt,elmname] {\(w\)};
            \draw[netedge] (v) to[bend left=25] node[midway, above=-2pt] {\(\scriptstyle \frac 1 2, 0\)} (w);
            \draw[netedge] (v) to[bend right=5] node[midway, above=-2pt] {\(\scriptstyle 1, -1\)} (w);
            \draw[netedge] (w) to[bend left=25] node[midway, below=-2pt] {\(\scriptstyle \frac 1 2, 0\)} (v);
            \draw[netedge] (C) to (v);
        \end{tikzpicture}}} \hfill
        \subcaptionbox[]{the polytope \(\Ord(\Gamma)\)\label{subfig:good-d}}[0.38\textwidth][c]{
        \begin{tikzpicture}[scale=1.2]
            \draw (-.1,0) -- (2,0) node [right] {\(x_v\)};
            \draw (0,-.1) -- (0,3) node [above] {\(x_w\)};
            \fill[black!6,draw=black,thick] (0,0) -- (2,1) -- (2,3) -- (1,2) -- cycle;
            \draw (0,1) node[left] {\(1\)};
            \draw (1,0) node[below] {\(1\)};
            \draw[blue, dashed, thick] (1,2) -- (2,2);
        \end{tikzpicture}} \hfill
        \subcaptionbox[]{the image \(\phi_\Gamma(\Ord(\Gamma))\)\label{subfig:good-image}}[0.38\textwidth][c]{
        \begin{tikzpicture}[scale=1.2]
            \draw (-.1,0) -- (2,0) node [right] {\(x_v\)};
            \draw (0,-.1) -- (0,3) node [above] {\(x_w\)};
            \fill[black!6,draw=black,thick] (0,0) -- (1.5,0) -- (1,1) -- (0,2) -- cycle;
            \draw (0,1) node[left] {\(1\)};
            \draw (1,0) node[below] {\(1\)};
            \draw[blue, dashed, thick] (0,1.5) -- (1,1);
        \end{tikzpicture}}
        \caption[Marked network and associated polytopes from \Cref{ex:cyclic-injective}]{The marked network \(\Gamma\) of \Cref{ex:cyclic-injective} with the associated distributive polytope and its bijective image under the transfer map.}
        \label{fig:good-example}
    \end{figure}
    The distributive polyhedron \(\Ord(\Gamma)\) is a quadrilateral given by
    the inequalities \(\tfrac 1 2 x_v \le x_w\), \(\tfrac 1 2 x_w \le x_v\),
    \(x_w - 1 \le x_v\) and \(x_v\le 2\) as shown in \Cref{subfig:good-d}.
    The transfer map for this network is given  by
    \begin{equation*}
        \phi_\Gamma \begin{pmatrix} x_v \\ x_w\end{pmatrix} =
        \begin{pmatrix} x_v - \max\{\frac 1 2 x_w, x_w-1\} \\ x_w - \frac 1 2 x_v \end{pmatrix}.
    \end{equation*}
    In this example, the transfer map is bijective and maps \(\Ord(\Gamma)\)
    to the anti-blocking polytope depicted in \Cref{subfig:good-image}.  The
    dashed line divides \(\Ord(\Gamma)\) into the two linearity regions of the
    transfer map.  We will come back to this example in \Cref{sec:ab-image}
    after constructing inverse transfer maps and describing the inequalities
    for \(\phi_\Gamma(\Ord(\Gamma))\).
\end{example}

As we have seen in \Cref{ex:cyclic-non-injective,ex:cyclic-injective}, some
cyclic networks lead to bijective transfer maps while others do not.  The
important difference in the two examples is the product of weights along the
cycles. This motivates the following definition.

\begin{definition} \label{def:lossygainy}
    Let \(\Gamma=(V\uplus A,E,\alpha,c,\lambda)\) be a marked network.  A
    \Def{walk} $W$ in \(\Gamma\) is a sequence 
    \begin{equation*}
        W \ = \ v_1 \xrightarrow{e_1} v_2 \xrightarrow{e_2} \cdots
        \xrightarrow{e_r} v_{r+1} \cdots.
    \end{equation*}
    If $W$ is finite, then its \Def{length} $|W|$ is the number of edges. The
    \Def{weight} of $W$ is 
    \begin{equation*}
        \alpha(W) \ \coloneqq \ \prod_{i=1}^{r} \alpha_{e_i}.
    \end{equation*}
    If all nodes are distinct, then $W$ is called a \Def{path}. If all nodes
    are distinct except for $v_{r+1}=v_1$, then we call $W$ a \Def{cycle}.
    In accordance with~\cite{FK11}, a cycle $C$ is called \Def{gainy} if
    \(\alpha(C) < 1\), \Def{lossy} if \(\alpha(C) > 1\) and \Def{breakeven} if
    \(\alpha(C) = 1\).  Finally, we call a marked network
    \Def{gainy}/\Def{lossy}/\Def{breakeven} if all cycles are
    gainy/lossy/breakeven.
\end{definition}

In the following section, we will show that the observation made in
\Cref{ex:cyclic-non-injective,ex:cyclic-injective} is true in general: when
\(\Gamma\) contains only gainy cycles, the transfer map is bijective.

\section{Gainy networks and infinite walks}\label{sec:gainy}

Throughout this section we assume that \(\Gamma=(V\uplus
A,E,\alpha,c,\lambda)\) is a gainy marked network such that every sink is
marked. Our goal is to construct an inverse to the transfer map
\(\phi_\Gamma\) and show that the image \(\phi_\Gamma(\Ord(\Gamma))\) is an
anti-blocking polyhedron by giving explicit inequalities determined by walks
in \(\Gamma\).

\begin{definition}
    To \(\Gamma\) associate the set \(\WW\) consisting of finite walks
    \begin{equation}
        v_1 \xrightarrow{e_1} v_2 \xrightarrow{e_2} \cdots \xrightarrow{e_r}
        v_{r+1} \quad\text{with \(v_i\in V\) for \(i\le r\) and \(v_{r+1}\in
        A\),}
        \label{eq:fin-walk}
    \end{equation}
    as well as infinite walks
    \begin{equation}
        v_1 \xrightarrow{e_1} v_2 \xrightarrow{e_2} v_3 \xrightarrow{e_3} \cdots
        \quad\text{with all \(v_i\in V\).}
        \label{eq:inf-walk}
    \end{equation}
    In particular, $A \subseteq \WW$, since walks of length $0$ are allowed.

    Given a walk \(W \in \WW\) starting in \(w\) and an edge
    \(v\lxrightarrow{e}w\) from an unmarked node \(v\in V\), denote by
    \(v\lxrightarrow{e}W\) the walk in \(\WW\) obtained by prepending the edge
    \(e\).
\end{definition}

In order to define the inverse transfer map, we want to associate to each walk
\(W\in\WW\) an affine-linear form \(\Sigma(W)\) on \(\R^V\) satisfying
\(\Sigma(a)(x) \coloneqq \lambda_a\) for all trivial walks at a marked element
\(a\in A\) and for all walks $W = v \to W'$ of positive length, the recursion
\begin{equation} \label{eq:Sigma-recursion}
    \Sigma(W)(x) \ = \ \alpha_e \Sigma(W')(x) + \left( x_v + c_e
    \right) \, .
\end{equation}

In order to see that $\Sigma$ is well-defined on infinite walks, we need the
following statement on convergence of infinite series.

\begin{prop} \label{prop:convergence}
    Let \(W \in \WW\) be an infinite walk as in \eqref{eq:inf-walk}.  The
    infinite series
    \begin{equation*}
        \sum_{k=1}^\infty \left( \prod_{j=1}^{k-1} \alpha_{e_j} \right) \left( x_{v_k} + c_{e_k} \right)
    \end{equation*}
    absolutely converges for all \(x\in\R^V\).
\end{prop}

\begin{proof}
    Since \(\Gamma\) has only finitely many nodes and edges, we have
    \(\left|x_{v_k}+c_{e_k}\right|\le M\) for some $M$.  It is therefore
    enough to show absolute convergence of \(\sum_{k=1}^\infty
    \prod_{j=1}^{k-1} \alpha_{e_j}\).  Using the root test, it is sufficient
    to show that
    \begin{equation*}
        \limsup_{k\to\infty} \left( \prod_{j=1}^k \alpha_{e_j} \right)^{\frac
        1 k} < 1.
    \end{equation*}

    Since \(\Gamma\) is finite, there are only finitely many paths and cycles
    and we can define
    \[
        a \coloneqq \max \left\{ \, \alpha(C)^{{\frac{1}{|C|}}} : C \text{ cycle}
        \,\right\} \quad\text{and} \quad 
        b \coloneqq \max \left\{ \, \alpha(P)^{{\frac{1}{|P|}}} : P \text{ path}
        \,\right\} \, .
    \]

    Now fix some \(k\in\N\) and consider the truncated walk 
    \begin{equation*}
        W^{(k)} \ = \ v_1 \xrightarrow{e_1} v_2 \xrightarrow{e_2} \cdots \xrightarrow{e_k} v_{k+1}
    \end{equation*}
    We may decompose \(W^{(k)}\) into a path from \(v_1\) to \(v_{k+1}\) and
    finitely many elementary cycles as depicted in \Cref{fig:walk-decomp}.
    \begin{figure}
        \centering
        \begin{tikzpicture}
                \node [posetelm] (0) at (0, 0) {};
                \node [posetelm] (1) at (1, 0) {};
                \node [posetelm] (2) at (2, 0) {};
                \node [posetelm] (3) at (3, 0) {};
                \node [posetelm] (4) at (4, -0.75) {};
                \node [posetelm] (5) at (4.5, -1.75) {};
                \node [posetelm] (6) at (3.5, -2.75) {};
                \node [posetelm] (7) at (5, -2.75) {};
                \node [posetelm] (8) at (6, -2.25) {};
                \node [posetelm] (9) at (5.5, -1.5) {};
                \node [posetelm] (10) at (2.25, -2.75) {};
                \node [posetelm] (11) at (1.5, -1.75) {};
                \node [posetelm] (12) at (1.75, -0.75) {};
                \node [posetelm] (13) at (4, 0) {};
                \node [posetelm] (14) at (5, 0) {};
                \node [posetelm] (15) at (6, 0) {};
                \draw[netedge]                    (0)  to node[midway, above] {\(\scriptstyle e_1\)} (1);
                \draw[netedge]                    (1)  to node[midway, above] {\(\scriptstyle e_2\)} (2);
                \draw[netedge]                    (2)  to node[midway, above] {\(\scriptstyle e_3\)} (3);
                \draw[netedge] [in=105, out=0]    (3)  to node[midway, below,xshift=-2pt,yshift=2pt] {\(\scriptstyle e_4\)} (4);
                \draw[netedge] [in=117, out=-75]  (4)  to node[midway, right,xshift=-2pt,yshift=2pt] {\(\scriptstyle e_5\)} (5);
                \draw[netedge] [in=150, out=-63]  (5)  to node[midway, right,xshift=-2pt,yshift=2pt] {\(\scriptstyle e_6\)} (7);
                \draw[netedge] [in=-90, out=-30]  (7)  to node[midway, above] {\(\scriptstyle e_7\)} (8);
                \draw[netedge] [in=-15, out=90]   (8)  to node[midway, right] {\(\scriptstyle e_8\)} (9);
                \draw[netedge] [in=30, out=165]   (9)  to node[midway, above] {\(\scriptstyle e_9\)} (5);
                \draw[netedge] [in=30, out=-150]  (5)  to node[midway, left,xshift=2pt,yshift=2pt] {\(\scriptstyle e_{10}\)} (6);
                \draw[netedge] [in=-15, out=-150] (6)  to node[midway, above] {\(\scriptstyle e_{11}\)} (10);
                \draw[netedge] [in=-90, out=165]  (10) to node[midway, right,xshift=-2pt,yshift=2pt] {\(\scriptstyle e_{12}\)} (11);
                \draw[netedge] [in=-120, out=90]  (11) to node[midway, right,xshift=-2pt,yshift=-2pt] {\(\scriptstyle e_{13}\)} (12);
                \draw[netedge] [in=180, out=60]   (12) to node[midway, below,xshift=2pt] {\(\scriptstyle e_{14}\)} (3);
                \draw[netedge]                    (3)  to node[midway, above] {\(\scriptstyle e_{15}\)} (13);
                \draw[netedge]                    (13) to node[midway, above] {\(\scriptstyle e_{16}\)} (14);
                \draw[netedge]                    (14) to node[midway, above] {\(\scriptstyle e_{17}\)} (15);
                
                \draw[thick,-triangle 45,decoration={snake,amplitude=1.7pt},decorate] (7,-1.5) -- (8,-1.5);

                \begin{scope}[shift={(8.5,0)}]
                    \node [posetelm] (0) at (0, 0) {};
                    \node [posetelm] (1) at (1, 0) {};
                    \node [posetelm] (2) at (2, 0) {};
                    \node [posetelm] (3) at (3, 0) {};
                    \node [posetelm] (13) at (4, 0) {};
                    \node [posetelm] (14) at (5, 0) {};
                    \node [posetelm] (15) at (6, 0) {};
                    \draw[netedge]                    (0)  to node[midway, above] {\(\scriptstyle e_1\)} (1);
                    \draw[netedge]                    (1)  to node[midway, above] {\(\scriptstyle e_2\)} (2);
                    \draw[netedge]                    (2)  to node[midway, above] {\(\scriptstyle e_3\)} (3);
                    \draw[netedge]                    (3)  to node[midway, above] {\(\scriptstyle e_{15}\)} (13);
                    \draw[netedge]                    (13) to node[midway, above] {\(\scriptstyle e_{16}\)} (14);
                    \draw[netedge]                    (14) to node[midway, above] {\(\scriptstyle e_{17}\)} (15);
                \end{scope}
                \begin{scope}[shift={(8.5,-.5)}]
                    \node [posetelm] (3) at (3, 0) {};
                    \node [posetelm] (4) at (4, -0.75) {};
                    \node [posetelm] (5) at (4.5, -1.75) {};
                    \node [posetelm] (6) at (3.5, -2.75) {};
                    \node [posetelm] (10) at (2.25, -2.75) {};
                    \node [posetelm] (11) at (1.5, -1.75) {};
                    \node [posetelm] (12) at (1.75, -0.75) {};
                    \draw[netedge] [in=105, out=0]    (3)  to node[midway, below,xshift=-2pt,yshift=2pt] {\(\scriptstyle e_4\)} (4);
                    \draw[netedge] [in=117, out=-75]  (4)  to node[midway, right,xshift=-2pt,yshift=2pt] {\(\scriptstyle e_5\)} (5);
                    \draw[netedge] [in=30, out=-150]  (5)  to node[midway, left,xshift=2pt,yshift=2pt] {\(\scriptstyle e_{10}\)} (6);
                    \draw[netedge] [in=-15, out=-150] (6)  to node[midway, above] {\(\scriptstyle e_{11}\)} (10);
                    \draw[netedge] [in=-90, out=165]  (10) to node[midway, right,xshift=-2pt,yshift=2pt] {\(\scriptstyle e_{12}\)} (11);
                    \draw[netedge] [in=-120, out=90]  (11) to node[midway, right,xshift=-2pt,yshift=-2pt] {\(\scriptstyle e_{13}\)} (12);
                    \draw[netedge] [in=180, out=60]   (12) to node[midway, below,xshift=2pt] {\(\scriptstyle e_{14}\)} (3);
                \end{scope}
                \begin{scope}[shift={(9,-.5)}]
                    \node [posetelm] (5) at (4.5, -1.75) {};
                    \node [posetelm] (7) at (5, -2.75) {};
                    \node [posetelm] (8) at (6, -2.25) {};
                    \node [posetelm] (9) at (5.5, -1.5) {};
                    \draw[netedge] [in=150, out=-63]  (5)  to node[midway, right,xshift=-2pt,yshift=2pt] {\(\scriptstyle e_6\)} (7);
                    \draw[netedge] [in=-90, out=-30]  (7)  to node[midway, above] {\(\scriptstyle e_7\)} (8);
                    \draw[netedge] [in=-15, out=90]   (8)  to node[midway, right] {\(\scriptstyle e_8\)} (9);
                    \draw[netedge] [in=30, out=165]   (9)  to node[midway, above] {\(\scriptstyle e_9\)} (5);
                \end{scope}
        \end{tikzpicture}        
        \caption{The decomposition of a finite walk into a path and cycles as
        used in the proof of \Cref{prop:convergence}.}
        \label{fig:walk-decomp}
    \end{figure}
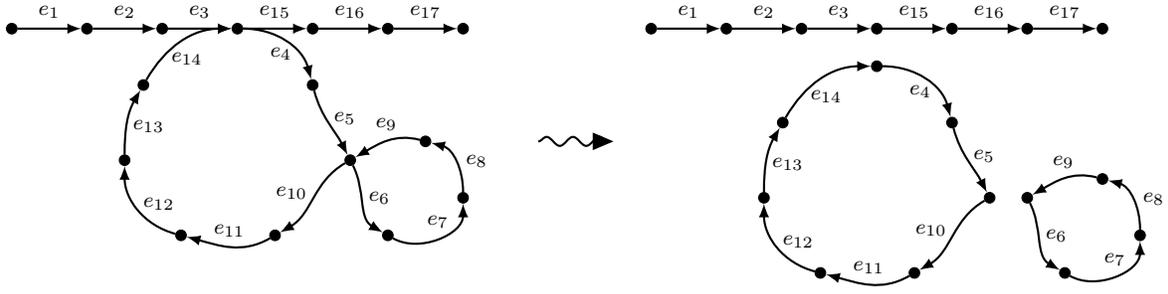
    If the path has \(t\le k\) edges, the total number of edges
    in the cycles is \(k-t\) and we obtain
    \begin{equation*}
        \alpha_{W^{(k)}} = \prod_{j=1}^k \alpha_{e_j} \le a^{k-t} b^t = \left( \frac b a \right)^t a^k. 
    \end{equation*}
    Let \(\ell \ge 0\) be the maximal length of a path in \(\Gamma\) and set
    \(c=\max(b/a, 1 )\) to obtain    
    \begin{equation*}
        \left( \prod_{j=1}^k \alpha_{e_j} \right)^{\frac 1 k} \le \left( c^\ell a^k \right)^{\frac 1 k} = c^{\frac \ell k} a
        \enskip\xrightarrow{k\to\infty}\enskip a.
    \end{equation*}
    Since all cycles in \(\Gamma\) are gainy by assumption, we have \(a<1\),
    finishing the proof.
\end{proof}

Using \Cref{prop:convergence}, we can define the desired linear forms.

\begin{definition} \label{def:sigma}
    For \(W \in \WW\) define an affine-linear form
    \(\Sigma(W)\colon\R^V\to\R\) as follows.  If \(W\) is a finite walk as in
    \eqref{eq:fin-walk}, set
    \begin{equation*}
        \Sigma(W)(x) \ \coloneqq \
        \sum_{k=1}^r \left( \prod_{j=1}^{k-1} \alpha_{e_j} \right) \left( x_{v_k} + c_{e_k} \right) +
        \left( \prod_{j=1}^r \alpha_{e_j} \right) x_{v_{r+1}}.
    \end{equation*}
    If \(W\) is an infinite walk as in \eqref{eq:inf-walk}, set
    \begin{equation*}
        \Sigma(W)(x) \ \coloneqq \
        \sum_{k=1}^\infty \left( \prod_{j=1}^{k-1} \alpha_{e_j} \right) \left( x_{v_k} + c_{e_k} \right).
    \end{equation*}
\end{definition}

By construction, the defined linear forms satisfy the recursion
\eqref{eq:Sigma-recursion}.  Indeed \eqref{eq:Sigma-recursion} together with
$\Sigma(a)(x)\coloneqq \lambda_a$ uniquely determines the linear forms
\(\Sigma(W)\) given the convergence in \Cref{prop:convergence}.

\begin{prop}
    For any \(x\in\R^V\) we have \(\sup_{W\in\WW} \Sigma(W)(x) < \infty\).
\end{prop}

\begin{proof}
    Let the constants \(M,a,b,c, \ell\) be given as in the proof of \Cref{prop:convergence}.
    For finite walks \(W\in\WW\) as in \eqref{eq:fin-walk}, we have
    \begin{equation*}
        \Sigma(W)(x) \ = \
        \sum_{k=1}^r \left( \prod_{j=1}^{k-1} \alpha_{e_j} \right) \left( x_{v_k} + c_{e_k} \right) +
        \left( \prod_{j=1}^r \alpha_{e_j} \right) x_{v_{r+1}}
        \ \le \ M \sum_{k=1}^{r+1} c^\ell a^{k-1}
        \ \le \ \frac{Mc^\ell}{1-a}.
    \end{equation*}
    Likewise, for infinite walks as in \eqref{eq:inf-walk}, we have
    \begin{equation*}
        \Sigma(W)(x) \ = \ \sum_{k=1}^\infty \left( \prod_{j=1}^{k-1}
        \alpha_{e_j} \right) \left( x_{v_k} + c_{e_k} \right) \ \le \ M
        \sum_{k=1}^\infty c^\ell a^{k-1} \ = \ \frac{Mc^\ell}{1-a}. \qedhere
    \end{equation*}
\end{proof}

\subsection{The inverse transfer map}\label{sec:inverse}

We are now ready to construct an inverse to the transfer map \(\phi_\Gamma\).
For \(v\in V\) denote by \(\WW_v\) the set of all walks \(\gamma\in\WW\) starting in \(v\).

\begin{thm} \label{thm:inverse-transfer}
    Let \(\Gamma=(V\uplus A,E,\alpha,c,\lambda)\) be a gainy marked network
    with all sinks marked. The transfer map \(\phi_\Gamma\colon\R^V\to\R^V\)
    is a piecewise-linear bijection with inverse
    \(\psi_\Gamma\colon\R^V\to\R^V\) given by
    \begin{equation*}
        \psi_\Gamma(y)_v \ \coloneqq \ \sup_{W\in\WW_v} \Sigma(W)(y).
    \end{equation*}
\end{thm}

Since part of the proof of \Cref{thm:inverse-transfer} will be relevant when we give a description of \(\phi_\Gamma(\Ord(\Gamma))\) below, we provide the following lemma first.

\begin{lem} \label{lem:psi-bound}
    For any \(x\in\R^V\) and \(v\in V\uplus A\) we have
    \begin{equation*}
        \sup_{W\in\WW_{v}} \Sigma(W)(\phi_\Gamma(x)) \le x_v.
    \end{equation*}
\end{lem}

\begin{proof}
    Let \(y=\phi_\Gamma(x)\) for \(x\in\R^V\).
    For a finite walk \(W\in\WW\) as in \eqref{eq:fin-walk} starting in \(v_1=v\), we have
    \begin{align*}
        \Sigma(W)(y) &=
        \sum_{k=1}^r \left( \prod_{j=1}^{k-1} \alpha_{e_j} \right) \left( y_{v_k} + c_{e_k} \right) +
        \left( \prod_{j=1}^r \alpha_{e_j} \right) y_{v_{r+1}}
        \\&=
        \sum_{k=1}^r \left( \prod_{j=1}^{k-1} \alpha_{e_j} \right) \left( x_{v_k} - \max_{v_k\ixrightarrow{e}w} (\alpha_e x_w + c_e) + c_{e_k} \right) +
        \left( \prod_{j=1}^r \alpha_{e_j} \right) x_{v_{r+1}}
        \\&\le
        \sum_{k=1}^r \left( \prod_{j=1}^{k-1} \alpha_{e_j} \right) \left( x_{v_k} - \alpha_{e_k} x_{v_{k+1}} \right) +
        \left( \prod_{j=1}^r \alpha_{e_j} \right) x_{v_{r+1}}
        \\&=
        x_{v_1} = x_v.
    \end{align*}
    For an infinite walk as in \eqref{eq:inf-walk} starting in \(v_1=v\), we have
    \begin{align*}
        \Sigma(W)(y) &=
        \sum_{k=1}^\infty \left( \prod_{j=1}^{k-1} \alpha_{e_j} \right) \left( y_{v_k} + c_{e_k} \right)
        \\&\le
        \sum_{k=1}^\infty \left( \prod_{j=1}^{k-1} \alpha_{e_j} \right) \left( x_{v_k} - \alpha_{e_k} x_{v_{k+1}} \right) = x_{v_1} = x_v. \qedhere
    \end{align*}
\end{proof}

\begin{proof}[{Proof of \Cref{thm:inverse-transfer}}]
    For \(v\in V\), all walks \(W\in\WW_v\) are of the form
    \(v\lxrightarrow{e}W'\) for an edge \(v\lxrightarrow{e}w\) and
    \(W'\in\WW_w\).
    Hence, by the recursive property \eqref{eq:Sigma-recursion} we have
    \begin{equation*}
        \Sigma(W)(y) = \alpha_e \Sigma(W')(y) + (y_v+c_e).
    \end{equation*}
    We conclude that \(\psi\) satisfies the recursion
    \begin{equation*}
        \psi_\Gamma(y)_v = y_v + \max_{v\ixrightarrow{e} w} \left(\alpha_e \psi_\Gamma(y)_w + c_e\right) \quad\text{for all \(v\in V\)}.
    \end{equation*}
    Comparing this to the definition of \(\phi_\Gamma\), we see that \(\phi_\Gamma\circ\psi_\Gamma\) is the identity on \(\R^V\).

    Regarding the composition \(\psi_\Gamma\circ\phi_\Gamma\), first note that \(\psi_\Gamma(\phi_\Gamma(x))_v \le x_v\) for all \(v\in V\) by \Cref{lem:psi-bound}.
    Hence, to show that \(\psi_\Gamma(\phi_\Gamma(x))_v = x_v\) for \(v\in
    V\), it is enough to construct a walk \(W\in\WW_{v}\) such that
    \(\Sigma(W)(\phi_\Gamma(x)) \ge x_v\). Let \(v_1=v\) and successively pick an edge \(v_k\lxrightarrow{e_k} v_{k+1}\) such that
    \begin{equation*}
        \alpha_{e_k} x_{v_{k+1}} + c_{e_k} = \max_{v_k\ixrightarrow{e}w} \left( \alpha_e x_w + c_e \right),
    \end{equation*}
    until either \(v_{k+1}\) is marked or \(v_{k+1}\) already appeared in \(\{v_1,\dots,v_k\}\).

    In the first case we constructed a finite walk \(W\in\WW_v\) as in \eqref{eq:fin-walk} satisfying
    \begin{align*}
        \Sigma(W)(\phi_\Gamma(x)) &=
        \sum_{k=1}^r \left( \prod_{j=1}^{k-1} \alpha_{e_j} \right) \left( \phi_\Gamma(x)_{v_k} + c_{e_k} \right) +
        \left( \prod_{j=1}^r \alpha_{e_j} \right) \phi_\Gamma(x)_{v_{r+1}}
        \\&=
        \sum_{k=1}^r \left( \prod_{j=1}^{k-1} \alpha_{e_j} \right) \left( x_{v_k} - \alpha_{e_k} x_{v_{k+1}} \right) +
        \left( \prod_{j=1}^r \alpha_{e_j} \right) x_{v_{r+1}} = x_{v_1} = x_v.
    \end{align*}

    In the second case, we ended at an unmarked element \(v_{r+1}=v_s\) for \(s\le r\).
    This yields an infinite walk \(W\in\WW_v\) of the form
    \begin{equation}
        v_1 \xrightarrow{e_1} \cdots \xrightarrow{e_{s-1}} v_s \xrightarrow{e_s} \cdots \xrightarrow{v_{r-1}} e_r \xrightarrow{e_r} v_s \xrightarrow{e_s} \cdots \xrightarrow{v_{r-1}} e_r \xrightarrow{e_r} v_s \xrightarrow{e_s} \cdots \label{eq:monocycle}
    \end{equation}
    That is, \(W\) walks from \(v_1\) to \(v_s\) and then infinitely often runs through the cycle
    \begin{equation*}
        v_s \xrightarrow{e_s} v_{s+1} \xrightarrow{e_{s+1}} \cdots \xrightarrow{v_{r-1}} e_r \xrightarrow{e_r} v_s.
    \end{equation*}
    Treating indices \(k>r\) accordingly, we obtain
    \begin{align*}
        \Sigma(W)(\phi_\Gamma(x)) &=
        \sum_{k=1}^\infty \left( \prod_{j=1}^{k-1} \alpha_{e_j} \right) \left( \phi_\Gamma(x)_{v_k} + c_{e_k} \right)
        \\&=
        \sum_{k=1}^\infty \left( \prod_{j=1}^{k-1} \alpha_{e_j} \right) \left( x_{v_k} - \alpha_{e_k} x_{v_{k+1}} \right) = x_{v_1} = x_v.
    \end{align*}

    In both cases \(\Sigma(W)(\phi_\Gamma(x))=x_v\) and we obtain \(\psi_\Gamma(\phi_\Gamma(x))_v=x_v\) as desired.
    We conclude that \(\phi_\Gamma\) and \(\psi_\Gamma\) are mutually inverse piecewise-linear self-maps of \(\R^V\).
\end{proof}

Inspecting the proof of \Cref{thm:inverse-transfer}, we see that only a finite
subset of \(\WW\) is necessary to define \(\psi_\Gamma\).  Namely, the paths
with only the last node marked and the infinite walks that keep repeating a
cycle after a finite number of steps as in \eqref{eq:monocycle}.  We will
refer to walks of the latter kind as \Def{monocycles} and denote them by $W =
P * C$, where $P$ is the path and $C$ the cycle. Note that only the end node
of $P$ is shared with $C$.  A visual representation of a monocycle can
be found in \Cref{fig:monocycle}.
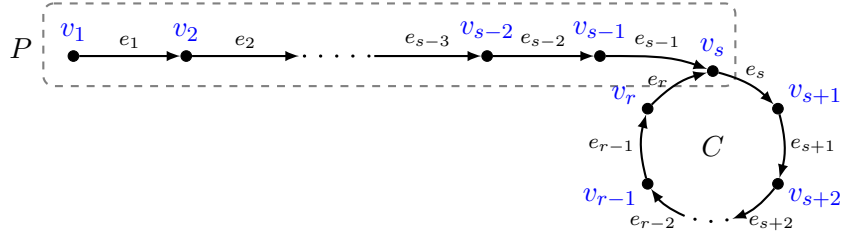
\begin{figure}
    \centering
    \begin{tikzpicture}
        \draw[rounded corners, dashed, thick, gray] (-8.9,1.9) -- (.3,1.9) -- (.3,.8) -- (-8.9,.8) -- cycle;

        \path (90:1) node[posetelm] (a) {} node[above=2pt,elmname] {\(v_s\)};
        \path (30:1) node[posetelm] (b) {} node[above right,elmname,yshift=-2pt] {\(v_{s+1}\)};
        \path (330:1) node[posetelm] (c) {} node[below right,elmname,yshift=2pt] {\(v_{s+2}\)};
        \path (210:1) node[posetelm] (e) {} node[below left, elmname,yshift=2pt] {\(v_{r-1}\)};
        \path (150:1) node[posetelm] (f) {} node[above left, elmname,yshift=-2pt] {\(v_r\)};

        \coordinate (x) at (-75:1);
        \coordinate (y) at (-112:1);

        \draw[netedge] (a) to[bend left=15] (b);
        \draw[netedge] (b) to[bend left=15] (c);
        \draw[netedge] (c) to[bend left=15] (x);
        \draw[dots] (x) to[bend left=15] (y);
        \draw[netedge] (y) to[bend left=15] (e);
        \draw[netedge] (e) to[bend left=15] (f);
        \draw[netedge] (f) to[bend left=15] (a);

        \path (60:.9) node[anchor=240] {\(\scriptstyle e_s\)};
        \path (00:.9) node[anchor=180] {\(\scriptstyle e_{s+1}\)};
        \path (-52.5:1) node[anchor=127.5] {\(\scriptstyle e_{s+2}\)};
        \path (-128:1) node[anchor=52] {\(\scriptstyle e_{r-2}\)};
        \path (180:.9) node[anchor=0] {\(\scriptstyle e_{r-1}\)};
        \path (130:.85) node[anchor=310] {\(\scriptstyle e_{r}\)};

        \path (-1.5,1.2) node[posetelm] (p) {} node[above=2pt,elmname] {\(v_{s-1}\)};
        \path (-3,1.2) node[posetelm] (p3) {} node[above=2pt,elmname] {\(v_{s-2}\)};
        \coordinate (q2) at (-4.5,1.2);
        \coordinate (q1) at (-5.5,1.2);
        \path (-7,1.2) node[posetelm] (p2) {} node[above=2pt,elmname] {\(v_{2}\)};
        \path (-8.5,1.2) node[posetelm] (p1) {} node[above=2pt,elmname] {\(v_{1}\)};

        \draw[netedge] (p) to[out=0, in=160] node[midway, above=-1pt]
        {\(\scriptstyle e_{s-1}\)} (a);
        \draw[netedge] (p3) to node[midway, above=-1pt] {\(\scriptstyle e_{s-2}\)} (p);
        \draw[netedge] (q2) to node[midway, above=-1pt] {\(\scriptstyle e_{s-3}\)} (p3);
        \draw[dots] (q1) to (q2);
        \draw[netedge] (p2) to node[midway, above=-1pt] {\(\scriptstyle e_2\)} (q1);
        \draw[netedge] (p1) to node[midway, above=-1pt] {\(\scriptstyle e_1\)} (p2);

        \node at (-9.2,1.3) {\(P\)};
        \node at (0,0) {\(C\)};
    \end{tikzpicture}
    \caption[A monocycle]{A monocycle. Note that the visible nodes are pairwise distinct.}
    \label{fig:monocycle}
\end{figure}

\begin{definition}
    Let \(\MW \subseteq \WW\) be the subset of walks \(W \in \WW\)
    such that \(W\) is either a path or a monocycle as in \eqref{eq:monocycle}
    with pairwise distinct \(v_1,\dots,v_r\).  For \(v\in V\), Denote by
    \(\MW_v=\MW\cap\WW_v\) the set of walks in \(\MW\)
    starting in \(v\).
\end{definition}

\begin{cor}
    The inverse transfer map \(\psi_\Gamma\colon\R^V\to\R^V\) is given by
    \begin{equation*}
        \psi_\Gamma(y)_v = \max_{W \in\MW_v} \Sigma(W )(y).
    \end{equation*}
\end{cor}

Since some of the \(W \in \MW_v\) appearing in this description of the inverse
transfer map might be monocycles, we want to give a finite expression for the
linear form \(\Sigma(W)\).

\begin{prop} \label{prop:monocycle-sum}
    Let \(W = P * C \in \MW\) be a monocycle with
    \begin{align*}
        P \ &= \ v_1 \xrightarrow{e_1} \cdots \xrightarrow{e_{s-1}} v_{s}, \\
        C \ &= \  v_s \xrightarrow{e_s} \cdots \xrightarrow{e_{r-1}} v_r \xrightarrow{e_r} v_s.
    \end{align*}
    Then for all \(x\in\R^V\) we have
    \begin{equation*}
        \Sigma(W)(x) \ = \ \sum_{k=1}^{s-1} \left( \prod_{j=1}^{k-1}
        \alpha_{e_j} \right) \left( x_{v_k} + c_{e_k} \right) +
        \frac{\alpha(P)}{1-\alpha(C)}
        \sum_{k=s}^r \left( \prod_{j=s}^{k-1} \alpha_{e_j} \right) \left( x_{v_k} + c_{e_k} \right).
    \end{equation*}
\end{prop}

\begin{proof}
    The infinite series in \Cref{def:sigma} yields that \(\Sigma(W)(x)\) is
    equal to
    \begin{equation} \label{eq:monocycle-series}
        \sum_{k=1}^{s-1} \left( \prod_{j=1}^{k-1} \alpha_{e_j} \right) \left( x_{v_k} + c_{e_k} \right)
        +
        \alpha(P)
        \sum_{l=0}^\infty \left[ \alpha(C)^l
        \sum_{k=s}^r \left( \prod_{j=s}^{k-1} \alpha_{e_j} \right) \left( x_{v_k} + c_{e_k} \right) \right].
    \end{equation}
    Since all cycles in \(\Gamma\) are gainy, we have \(\alpha(C)<1\) and the
    geometric series \(\sum_{l=0}^\infty \alpha(C)^l\) converges to
    \((1-\alpha(C))^{-1}\).
\end{proof}

Let $x \in \R^V$. Now for every $v \in V$, select an edge $v \to w$ which
attains the maximum in the definition of $\phi_\Gamma(x)_v$
in~\eqref{eqn:phi}.  It becomes clear from our discussion that this yields a
marked subnetwork $\Gamma_x$ composed of paths and monocycles. More precisely,
deleting the cycles, leaves a \emph{rooted} forest, that is, an acyclic
digraph in which every nodes has one edge pointing out. This network realizes
$\phi_\Gamma$ as an affine-linear function at $x$. The matrix $B = B(\Gamma_x)
\in \R^{V \times V}$ with $B_{ww} = 1$ and 
\[
    B_{wv} \ = \ -\alpha_e \quad \text{if } w \xrightarrow{e} v
\]
and $0$ otherwise determines the linear part of $\phi_\Gamma$ at $x$. The
determinant of $B$ is $\prod_C (1-\alpha(C))$ where $C$ ranges over the cycles
in $\Gamma_x$.

\begin{cor}\label{cor:vol_lattice}
    Let $\Gamma = (V\uplus A,E,\alpha,c,\lambda)$ be a gainy marked network
    with all sinks marked.  If $\alpha(C) = 2$ for all cycles $C$, then
    $\phi_\Gamma$ is volume preserving. If the weights $\alpha$ and $c$ are
    integral, then $\phi_\Gamma$ is lattice-preserving if and only if for
    every cycle there is a unique edge $e'$ with weight $\alpha_{e'}= 2$ and
    all other edges have weight $\alpha_e=1$.
\end{cor}

\section{Anti-blocking images}\label{sec:ab-image}

In the previous section, we showed that distributive polyhedra given by gainy
marked networks with at least all sinks marked admit a piecewise-linear
bijective transfer map \(\phi_\Gamma\colon\R^V\to\R^V\) analogous to the
transfer map for (marked) order polytopes.  In this section we keep the same
premise and focus on the image \(\phi_\Gamma(\Ord(\Gamma))\).  We show that it
is an anti-blocking polyhedron with describing inequalities given by the walks
in \(\MW\), similar to the chain polytope being described by
inequalities given by chains in the poset.

\begin{definition} \label{def:network-ab}
    Let \(\Gamma=(V\uplus A,E,\alpha,c,\lambda)\) be a gainy marked network
    with at least all sinks marked.  The polyhedron \(\Chain(\Gamma)\) is the
    set of all \(y\in\R^V\) with $y \ge 0$ and 
    \begin{equation}\label{eqn:monocyc_ineq}
                \alpha_e \Sigma(W)(y) + c_e \ \le \ \lambda_a 
    \end{equation}
    for each walk \(a \lxrightarrow{e} W\) with $a\in A$ and $W\in\MW$.
\end{definition}

\begin{thm}\label{thm:PL_homeo}
    Let \(\Gamma=(V\uplus A,E,\alpha,c,\lambda)\) be a gainy marked network
    with at least all sinks marked.  The transfer map \(\phi_\Gamma\)
    restricts to a piecewise-linear homeomorphism
    \(\Ord(\Gamma)\to\Chain(\Gamma)\).
\end{thm}

\begin{proof}
    To show that \(\phi_\Gamma(\Ord(\Gamma))\subseteq\Chain(\Gamma)\), let
    \(y=\phi_\Gamma(x)\) for \(x\in\Ord(\Gamma)\). By definition of
    \(\phi_\Gamma\) we have \(y_v\ge 0\) for \(v\in V\).  Now let
    \(a\lxrightarrow{e}W\) be a walk with \(a \in A\) and \(W\in\MW\).  It follows
    from \Cref{lem:psi-bound} that
    \begin{equation*}
        \alpha_e \Sigma(W)(y) + c_e \le \alpha_e x_v +c_e \le x_a = \lambda_a.
    \end{equation*}
    Now let \(y\) be any point in \(\Chain(\Gamma)\) and let \(x=\psi_\Gamma(y)\).
    For any edge \(v\lxrightarrow{e}w\) we have to show that \(\alpha_e x_w + c_e\le x_v\).
    Let \(W\in\MW_w\) be a walk starting in \(w\) constructed as
    in the proof of \Cref{thm:inverse-transfer} such that \(\Sigma(W)(y)=x_w\).
    If \(v\in V\), we can again appeal to \Cref{lem:psi-bound} together with $y \ge 0$ to
    obtain
    \begin{equation*}
        \alpha_e x_w + c_e = \alpha_e \Sigma(W)(y) + c_e
        = \Sigma(v\xrightarrow{e}W)(y) - x_v \le \Sigma(v\xrightarrow{e}W)(y) \le x_v.
    \end{equation*}
    Otherwise, if \(v\in A\), the walk \(v\lxrightarrow{e}W\) appears in
    \Cref{def:network-ab}, so that
    \begin{equation*}
        \alpha_e x_w + c_e = \alpha_e \Sigma(W)(y) + c_e \le \lambda_v = x_v.
        \qedhere
    \end{equation*}
\end{proof}

Fulkerson~\cite{Ful71} introduced anti-blocking polyhedra and gave the
following characterization.

\begin{prop}[\cite{Ful71}] \label{prop:AB}
    A polyhedron $Q \subseteq \R^d_{\ge 0}$ is anti-blocking if and only if
    there are $a_1,\dots,a_m \in \R^d_{\ge 0}$ and $b_1,\dots,b_m \in \R_{\ge
    0}$ such
    that 
    \[
    Q \ = \ \{ x \in \R^d_{\ge 0} : a_i^t x \le b_i \text{ for } i=1,\dots,m
    \} \, .
    \]
\end{prop}

This description allows us to prove the following.

\begin{cor}
    The polyhedron \(\Chain(\Gamma)\) is anti-blocking.
\end{cor}

\begin{proof}
    By definition $\Chain(\Gamma) \subseteq \R^V_{\ge 0}$.
    Furthermore, the coefficients in
    an inequality \(\alpha_e \Sigma(\gamma)(y) + c_e \le \lambda_a\) are all
    non-negative: for finite walks they are just finite products of edge
    weights \(\alpha_{e'}\) while for monocycles some of them are multiplied
    by the positive factor \(\alpha(P)/(1-\alpha(C))\) as described in
    \Cref{prop:monocycle-sum}.
\end{proof}

\begin{example}[{continuation of \Cref{ex:cyclic-injective}}]
    Recall the marked network \(\Gamma\) with two unmarked nodes depicted in
    \Cref{fig:good-example} together with the distributive polytope
    \(\Ord(\Gamma)\) and its anti-blocking image now denoted by
    \(\Chain(\Gamma)\). We label the three edges between $v$ and $w$ from top
    to bottom by $e,f,g$.

    Since \(\Gamma\) does not have marked nodes with incoming edges and all
    cycles contain only unmarked nodes, the set of monocycles \(\MW\)
    is given by the cycles with trivial acyclic beginning:
    \begin{align*}
        W_1 = v\xrightarrow{e} w\xrightarrow{g} v\xrightarrow{e} w
        \xrightarrow{g} \cdots \quad & \quad
        W_2 = v\xrightarrow{f} w\xrightarrow{g} v\xrightarrow{f} w
        \xrightarrow{g} \cdots \\
        W_3 = w\xrightarrow{g} v\xrightarrow{e} w \xrightarrow{g} v
        \xrightarrow{e} \cdots \quad & \quad
        W_4 = w\xrightarrow{g} v\xrightarrow{f} w \xrightarrow{g} v
        \xrightarrow{f} \cdots
    \end{align*}
    From \Cref{prop:monocycle-sum} with trivial acyclic beginning (\(s=1\)) we obtain
    \begin{align*}
        \Sigma(W_1)(y) = \tfrac 4 3 x_v + \tfrac 2 3 x_w \quad & \quad
        \Sigma(W_2)(y) = 2 x_v + 2 x_w - 2 \\
        \Sigma(W_3)(y) = \tfrac 2 3 x_v + \tfrac 4 3 x_w \quad & \quad
        \Sigma(W_4)(y) = x_v + 2 x_w - 1
    \end{align*}
    Hence, the inverse transfer map on \(\R^{V}\) is given by
    \begin{equation*}
        \psi_\Gamma \begin{pmatrix} y_v \\ y_w \end{pmatrix} =
        \begin{pmatrix}
            \max \{ \frac 4 3 x_v + \frac 2 3 x_w, 2 x_v + 2 x_w - 2 \} \\
            \max \{ \frac 2 3 x_v + \frac 4 3 x_w, x_v + 2 x_w - 1\}
        \end{pmatrix}.
    \end{equation*}
    Note that the linearity regions are the two half-spaces given by the
    hyperplane \(\tfrac 1 3 x_v + \tfrac 2 3 x_w = 1\) containing the dashed
    line in \Cref{subfig:good-image}.
    For the anti-blocking image \(\Chain(\Gamma)\) the only walks appearing in
    \Cref{def:network-ab} are \(2\to W_1\) and \(2\to W_2\) giving
    inequalities
    \[
        \tfrac 4 3 x_v + \tfrac 2 3 x_w \le 2 \quad\text{and}\quad 
        2 x_v + 2 x_w \le 4.
    \]
    These correspond to the two non-trivial facets in \Cref{subfig:good-image}.
\end{example}

In \Cref{ex:cyclic-non-injective}, where we have a lossy cycle and the
transfer map is not injective, the image was still an anti-blocking polytope.
However, this is not true in general: in the following example we have a lossy
cycle, an injective transfer map nevertheless, but the image
\(\phi_\Gamma(\Ord(\Gamma))\) is not anti-blocking.
\begin{example} \label{ex:notab}
    Let \(\Gamma\) be the marked network shown in \Cref{subfig:notab-network}.
    \begin{figure}
        \centering
        \subcaptionbox[]{a network \(\Gamma\)\label{subfig:notab-network}}[0.18\textwidth][c]{
        \raisebox{2em}{
        \begin{tikzpicture}[xscale=.6,yscale=1.5]
            \path (0,2) node[posetelmm] (C) {} node[above=2pt,marking] {\(3\)};
            \path (1,1) node[posetelm] (A) {} node[right=2pt,elmname] {\(w\)};
            \path (-1,1) node[posetelm] (B) {} node[left=2pt,elmname] {\(v\)};
            \draw[netedge] (A) to[bend left=10] node[midway, below=-2pt] {\(\scriptstyle 2,-4\)} (B);
            \draw[netedge] (B) to[bend left=10] node[midway, above=-2pt] {\(\scriptstyle 2,-4\)} (A);
            \draw[netedge] (C) to[bend right=30] (B);
            \draw[netedge] (C) to[bend left=30] (A);
        \end{tikzpicture}}} \hfill
        \subcaptionbox[]{the polyhedron \(\Ord(\Gamma)\)\label{subfig:notab-d}}[0.38\textwidth][c]{
        \begin{tikzpicture}[scale=1.2]
            \draw[thin, black!20] (0,3) -- (3,3) -- (3,0);
            \fill[black!6] (-.5,1.75) -- (2,3) -- (3,3) -- (3,2) -- (1.75,-.5) -- (-.5,-.5) -- cycle;
            \draw[thick] (-.5,1.75) -- (2,3) -- (3,3) -- (3,2) -- (1.75,-.5);
            \draw (-.5,0) -- (3.5,0) node [right] {\(x_v\)};
            \draw (0,-.5) -- (0,3.5) node [above] {\(x_w\)};
            \draw[blue,dashed, thick] (0,2) -- (3,2);
            \draw[blue,dashed, thick] (2,0) -- (2,3);
            \draw (0,2) node[left,yshift=2pt] {\(2\)};
            \draw (0,3) node[left] {\(3\)};
            \draw (2,0) node[below,xshift=1pt] {\(2\)};
            \draw (3,0) node[below] {\(3\)};
        \end{tikzpicture}} \hfill
        \subcaptionbox[]{the image \(\phi_\Gamma(\Ord(\Gamma))\)\label{subfig:notab-image}}[0.38\textwidth][c]{
        \begin{tikzpicture}[scale=.65]
            \fill[black!6] (0,6.5) -- (6.5,6.5) -- (6.5,0) -- (3,0) -- (1,1) -- (0,3) -- cycle;
            \draw[thin, black!20] (0,2) -- (2,2) -- (2,0);
            \draw[thin, black!20] (0,1) -- (1,1) -- (1,0);
            \draw[thick] (6.5,0) -- (3,0) -- (1,1) -- (0,3) -- (0,6.5);
            \draw (-.5,0) -- (6.5,0) node [right] {\(x_v\)};
            \draw (0,-.5) -- (0,6.5) node [above] {\(x_w\)};
            \draw[blue,dashed, thick] (0,6) -- (3,0);
            \draw[blue,dashed, thick] (0,3) -- (6,0);
            \draw (0,1) node[left] {\(1\)};
            \draw (0,2) node[left] {\(2\)};
            \draw (0,3) node[left] {\(3\)};
            \draw (0,6) node[left] {\(6\)};
            \draw (1,0) node[below] {\(1\)};
            \draw (2,0) node[below] {\(2\)};
            \draw (3,0) node[below] {\(3\)};
            \draw (6,0) node[below] {\(6\)};
        \end{tikzpicture}}
        \caption[Marked network and associated polyhedra from \Cref{ex:notab}]{The marked network \(\Gamma\) of \Cref{ex:notab} with the associated distributive polyhedron and its non-anti-blocking image under the transfer map.}
        \label{fig:notab-example}
    \end{figure}
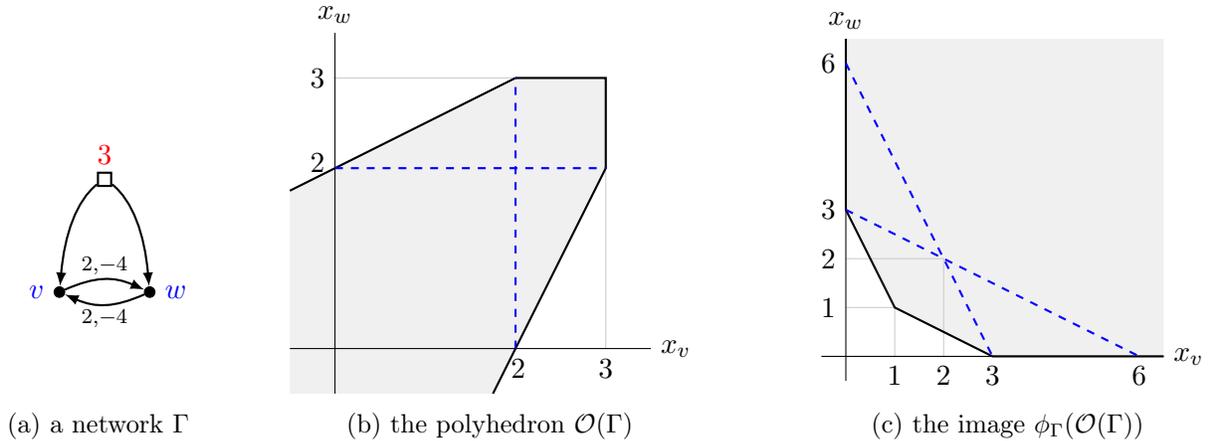
    The distributive polyhedron \(\Ord(\Gamma)\) is the unbounded polyhedron in \Cref{subfig:notab-d} given by the inequalities \(2x_v-4\le x_w\), \(2x_w-4\le x_v\), \(x_v\le 3\) and \(x_w\le 3\).
    The transfer map is given by
    \begin{equation*}
        \phi_\Gamma \begin{pmatrix} x_v \\ x_w\end{pmatrix} =
        \begin{pmatrix} x_v - 2x_w + 4 \\ x_w - 2x_v + 4\end{pmatrix}.
    \end{equation*}
    Thus, the image \(\phi_\Gamma(\Ord(\Gamma))\) is the polyhedron given by
    inequalities \(0\le y_v\), \(0\le y_w\), \(y_v + 2 y_w \ge 3\) and \(2 y_v
    + y_w \ge 3\).  It is depicted in \Cref{subfig:notab-image} and is not an
    anti-blocking polyhedron.  In fact it is what is called a \Def{blocking
    polyhedron} in \cite{Ful71}: it is given given by inequalities \(x_i\ge
    0\) for all coordinates together with inequalities of the form $a^tx \ge
    1$  with $a \in \R^n_{\ge0}$.
\end{example}

\subsection{Duality}

Let $Q \subseteq \R^n$ be a distributive polyhedron. It follows from the
definition that $-Q$ is distributive as well. If $Q = \Ord(\Gamma)$, then $-Q
= \Ord(\Gamma^\op)$, where $\Gamma^\op = (V \uplus A, E',
\alpha',c',\lambda')$ is the \Def{opposite network} with edges $w\lxrightarrow{e'}v$
for each edge $v\lxrightarrow{e}w$ in $\Gamma$, 
weights $\alpha_{e'} =
\frac{1}{\alpha_e}$, $c_{e'} = \frac{c_e}{\alpha_e}$, and $\lambda' =
-\lambda$. If $(P,\preceq)$ is a poset, then $-\Ord(P)$ is, up to a
translation, the order polytope $\Ord(P^\op)$, where $P^\op$ is the opposite
poset.

If $\Gamma$ is a network with all \emph{sources} marked, then
$\phi^{\op}_\Gamma \colon \R^V \to \R^V$ given by $\phi^{\op}_\Gamma(x) \coloneqq
\phi_{\Gamma^\op}(-x)$ is a piecewise-linear map. More precisely, it is given by
\begin{equation*}
    \phi^\op_\Gamma(x)_v  \ = \ -x_v + \min_{w\ixrightarrow{e}v} \left(
    \tfrac{1}{\alpha_e} x_w - \tfrac{c_e}{\alpha_e} \right) .
\end{equation*}
If $\Gamma$ has only lossy cycles, then $\phi^\op_\Gamma$ is bijective and
restricts to a homeomorphism $\Ord(\Gamma) \to \Chain(\Gamma^\op)$.

When $\Gamma$ is acyclic and both all sinks and all sources are marked, we can compare
the anti-blocking polyhedra $\Chain(\Gamma)$ and $\Chain(\Gamma^\op)$. If $\Gamma$ is
the Hasse diagram of a poset, we have $\Chain(\Gamma)=\Chain(\Gamma^\op)$ as a consequence
of the opposite poset having the same comparability graph. By comparing the
defining inequalities of the two polyhedra in the general case, we can see that
this observation still holds for arbitrary acyclic marked networks with all
sinks and sources marked.

\section{Applications and questions}\label{sec:apps}

\subsection{Cayley polytopes}
Recall that the Cayley polytope \(C_n\) is the distributive polytope
\(\Ord(\Gamma)\) associated to the marked network in
\Cref{fig:cayley-network}.  The geometric bijection in~\cite{KP2} is a linear
transformation \(\phi^{-1} \colon C_n \to \mathbf Y_n\), where $\mathbf{Y}_n$
is an anti-blocking polytope defined in the introduction. This map is exactly
the transfer map \(\psi^\op_\Gamma \colon \Chain(\Gamma)\to\Ord(\Gamma)\).

\subsection{Lecture hall order cones and polytopes}
The $s$-lecture hall cones and polytopes of Bousquet-M\'elou and
Eriksson~\cite{BE97, BE97a} and Stanley's
$P$-partitions~\cite[Sect.~3.15]{EC1} were elegantly combined in~\cite{BL} to
\emph{lecture hall order cones/polytopes}. Here we briefly sketch a
generalization to a marked version. Let \((P,\preceq, \lambda)\) be a
marked poset with $\lambda \in \R^A$ for $A\subseteq P$. For any $s \in \R^P_{>0}$, define  the
\Def{marked lecture hall order polyhedron} \(\Ord(P,\lambda,s)\) as the set of
points $x \in \R^{P\setminus A}$
\begin{equation*}
    \frac{x_p}{s_p} \ \le \ \frac{x_q}{s_q} \quad\text{for \(p\prec q\)},
\end{equation*}
where we set $x_a = s_a \lambda_a$ for $a \in A$. If $s \equiv 1$, then
\(\Ord(P,\lambda,s)\) is the marked order polyhedron \(\Ord(P,\lambda)\).  When
\(P\) is the linear poset \(\zerohat\prec p_1\prec\cdots\prec p_n\) and
\(\lambda_{\zerohat}=0\), we recover the \(s\)-lecture hall cones and adding a
maximal element $\onehat$ with marking \(\lambda_{\onehat}=1\) we get the
\(s\)-lecture hall polytopes.

Note that \(\Ord(P,\lambda,s)=\Ord(\Gamma)\) for the marked network given by
the Hasse diagram of $P$ with edge weights $c\equiv 0$ and
$\alpha_e=\tfrac{s_q}{s_p}$ for an edge $e$ given by a covering relation
$p\prec q$. We may also express \(\Ord(P,\lambda,s)\) as a linear
transformation \(T_s(\Ord(P,\lambda))\) of the usual marked order polyhedron,
where \(T_s(x)_p = s_p x_p\). This transformation is compatible with the
transfer maps associated to \(\Ord(P,\lambda)\) and
\(\Ord(P,\lambda,s)=\Ord(\Gamma)\) in the sense that $T_s\circ
\phi_{(P,\lambda)} = \phi_\Gamma \circ T_s$. If $\lambda$ and $s$ are integral, then the
marked lecture hall order polytopes are lattice polytopes. Furthermore,
if \(s\) satisfies $s_p \mid s_q$ for $p \prec q$, then the transfer map
is lattice preserving.

\subsection{Coordinates in polytopes} The geometric
reformulation~\eqref{eqn:geom_phi} admits the following generalization: Given
a polyhedron $Q \subseteq \R^d$ and vectors $U = (u_1,\dots,u_m) \in \R^{d\times
m}$. Define $\phi_{Q,U} \colon Q \to \R^m$ by
\[
    \phi_{Q,U}(x)_i  \ \coloneqq \ \max( \mu \ge 0 : x - \mu v_i \in Q ) \, .
\]

\begin{quest}
    For which $(Q,U)$ is $\phi_{Q,U}$ injective? When is the image convex?
\end{quest}

If $\phi_{Q,U}$ is injective and convex, then its image gives a representation
of $Q$ up to translation, akin to its \emph{slack representation};
see~\cite[Sect.~3.2]{Kaibel}. Our results show that $\phi_\Gamma$ yields a class
of examples for gainy networks. However, \Cref{ex:notab} shows that even for
some distributive polyhedra associated to marked networks with non-gainy
cycles the transfer map can still be injective.

\subsection{Continuous families}
In~\cite{XinFourier}, Fang and Fourier generalized the marked poset polytopes
\(\Ord(P,\lambda)\) and \(\Chain(P,\lambda)\) to a discrete family of marked
chain-order polytopes \(\Ord_{C,O}(P,\lambda)\).  It is parametrized by
partitions \(P\setminus A = C\uplus O\) such that $C=\varnothing$ yields the
order polytope, which is distributive, and $O=\varnothing$ yields the chain
polytope, which is anti-blocking.  When both $C$ and $O$ are non-empty,
\(\Ord_{C,O}(P,\lambda)\) is neither distributive nor anti-blocking in
general.

\begin{definition}
    Let \(D\) and \(A\) be finite sets. A polyhedron \(Q\subseteq
    \R^D\times\R_{\ge 0}^A\) is called \Def{mixed distributive anti-blocking}
    if it satisfies the following properties:
    \begin{enumerate}[noitemsep,label=\roman*)]
        \item given \((x,z)\in Q\) and \((y,z)\in Q\), we have \((x \wedge y,z)\in Q\) and \((x \vee y,z)\in Q\),
        \item when \((x,z)\in Q\) and \(0\le y\le z\), then \((x,y)\in Q\).
    \end{enumerate}
\end{definition}
When \(A=\varnothing\) or \(D=\varnothing\), this recovers the notions of distributive and anti-blocking polyhedra, respectively.
The marked chain-order polytopes are then mixed distributive anti-blocking with respect to the decomposition \(\R^{P\setminus A} = \R^O \times \R^C\).

This discrete family of marked chain-order polytopes has been embedded into a
continuous family of polytopes \(\Ord_t(P,\lambda)\) parametrized by
$t\in[0,1]^{P\setminus A}$ in \cite{FFLP}.  The marked chain-order polytopes
are obtained for characteristic functions $t=\chi_C$. These polytopes are all
obtained as images of the marked order polytope \(\Ord(P,\lambda)\) under
parametrized transfer maps
\[
    \phi_t(x)_p \ \coloneqq\ x_p - t_p\cdot \max_{q\prec p} x_q.
\]

Hence, it is natural to ask whether we can obtain an analogous continuous family of polyhedra associated to marked networks.
\begin{quest}
    Does introducing a parameter \(t\in[0,1]^V\) in the transfer map of
    distributive polyhedra associated to gainy marked networks with marked
    sinks yield a continuous family of polyhedra such that
    \begin{enumerate}[label=\roman*)]
        \item the combinatorial type of the images is constant along relative
            interiors of the parametrizing hypercube and
        \item the polyhedra at the vertices of the hypercube are mixed
            distributive anti-blocking?
    \end{enumerate}
\end{quest}

\subsection{Domains of linearity, faces, Minkowski summands}
At the end of \Cref{sec:inverse}, we gave an idea of the domains of
linearity of $\phi_\Gamma$. They are related to rooted forests with cycles
attached to some leafs. Stanley~\cite{Stanley} considered a refined
subdivision of $\Ord(P)$ that had the property of being unimodular.  For
marked order polytopes a corresponding subdivision was described in~\cite{JS}
in terms of products of dilated unimodular simplices. In the general case with
arbitrary weights it is not clear if such fine subdivisions exist.
\begin{quest}
    Do distributive polyhedra admit a natural subdivision into products of
    simplices on which the transfer map is linear?
\end{quest}

The face structure of marked order polyhedra can be described by so-called
face partitions~\cite{Stanley,JS,Pegel}. The question of describing the
vertices of $\Ord(\Gamma)$ was also raised in~\cite{FK11}.

\begin{quest}
    Give a combinatorial description of the faces of $\Ord(\Gamma)$ in terms
    of the underlying network.
\end{quest}

In~\cite{Fou16,Pegel}, a marked poset is called \Def{regular} if the
inequalities derived from the cover relations are irredundant (or
facet-defining).

\begin{quest}
    When is a marked network \emph{regular}?
\end{quest}

Finally, polyhedra may be decomposed into Minkowski summands. For marked order
polyhedra this was done in~\cite{JS,Pegel}, for marked chain-order polyhedra
in \cite{XinFourier,FFP}.

\begin{quest}
    Is there a Minkowski sum decomposition of distributive polyhedra similar
    to the one for marked order polyhedra and marked chain-order polyhedra?
\end{quest}

\bibliographystyle{siam}
\bibliography{DistribPolyhedra}

\end{document}